\documentclass[a4paper,11pt,reqno]{amsart}

\usepackage{amsmath}
\usepackage{amsfonts}
\usepackage{amssymb}
\usepackage{amsthm}
\usepackage[shortlabels]{enumitem}
\usepackage{graphicx}
\usepackage{color}
\usepackage{dsfont}
\usepackage{inputenc}
\usepackage{MnSymbol}
\usepackage{enumitem}
\usepackage{extpfeil}
\usepackage{mathtools}
\usepackage{url}
\usepackage[margin=1in]{geometry}
\usepackage{fancyhdr}
\usepackage[all,cmtip]{xy}

\numberwithin{equation}{section}

\title{Kodaira fibrations,  K\"ahler groups, and finiteness properties}

\author{Martin R. Bridson}
\address{Martin R. Bridson, Mathematical Institute, Andrew Wiles Building, University of Oxford, Oxford OX2 6GG, EU}
\email{bridson@maths.ox.ac.uk}

\author{Claudio Llosa Isenrich}
\address{Claudio Llosa Isenrich, Laboratoire de Math\' ematiques d'Orsay, Univ. Paris-Sud, CNRS, Universit\' e Paris-Saclay, 91405 Orsay, France}
\email{claudio.llosa-isenrich@math.u-psud.fr}

\thanks{Bridson is funded by a Wolfson Research Merit Award from the Royal Society.
 Llosa Isenrich is supported by an EPSRC Research Studentship and by the German National Academic Foundation}
 


\begin{document}

\newcommand{\QQ}{{\mathds Q}}
\newcommand{\RR}{{\mathds R}}
\newcommand{\NN}{{\mathds N}}
\newcommand{\ZZ}{{\mathds Z}}
\newcommand{\del}{{\partial}}
\def\C{{\mathds C}}
\def\S{\Sigma}
\def\F{{\mathds F}}
\def\FF{{\mathcal F}}
\def\nn{{\bf N}}
\def\aut{{\rm{Aut}}}
\def\inn{{\rm{Inn}}}
\def\out{{\rm{Out}}}
\def\Mod{{\rm{Mod}}}
\def\isom{{\rm{Isom}}}
\def\mcg{{\rm{MCG}}}
\def\ker{{\rm{ker }}}
\def\G{\Gamma}
\def\g{\gamma}
\def\L{\Lambda}
\def\Z{{\mathds{Z}}}
\def\H{{\mathds{H}}}

\theoremstyle{plain}
\newtheorem{theorem}{Theorem}[section]
\newtheorem{acknowledgements}[theorem]{Acknowledgements}
\newtheorem{claim}[theorem]{Claim}
\newtheorem{conjecture}[theorem]{Conjecture}
\newtheorem{corollary}[theorem]{Corollary}
\newtheorem{exercise}[theorem]{Exercise}
\newtheorem{lemma}[theorem]{Lemma}
\newtheorem{proposition}[theorem]{Proposition}
\newtheorem{question}{Question}
\newtheorem{addendum}[theorem]{Addendum}

\theoremstyle{definition}
\newtheorem{remark}[theorem]{Remark}
\newtheorem*{acknowledgements*}{Acknowledgements}
\newtheorem{example}[theorem]{Example}
\newtheorem{definition}[theorem]{Definition}

\renewcommand{\proofname}{Proof}

\begin{abstract}
We construct classes of K\"ahler groups that do not have finite classifying spaces 
and are not commensurable to subdirect products of surface groups. Each of these groups is
the fundamental group of the generic fibre of a holomorphic map 
from a product of Kodaira fibrations onto an elliptic curve. 
\end{abstract}

\maketitle

\section{Introduction}

A \textit{K\"ahler group} is a group that can be realised as the fundamental group of a compact K\"ahler manifold. The question of which finitely presented groups are K\"ahler was raised by Serre in the 1950s. It has been a topic of active research ever since, but a putative classification remains a distant prospect and  constructions of novel examples are surprisingly rare.
For an overview of what is known see \cite{ABCKT-95} and \cite{Bur-10}.

In this paper we present a new technique for constructing K\"ahler groups. By applying this technique to Kodaira fibrations, we obtain K\"ahler groups
that do not have finite classifying spaces. We also develop a criterion for deciding when these groups are commensurable with residually free groups.

A group $G$ is of \textit{type} $\mathcal{F}_r$ if it has a classifying space $K(G,1)$ with finite $r$-skeleton. 
The first example of a finitely presented group that is not of
type $\FF_3$ was given by Stallings in 1963 \cite{Sta-63}. His example is a subgroup of a direct product of three free groups. Bieri subsequently constructed, for each positive integer $n$, a subgroup $B_n$ of a direct product of $n$ free groups such that $B_n$ is of type $\FF_{n-1}$ but not of type $\FF_n$; each $B_n$ is the
kernel of a map from the ambient direct product to an abelian group. The study of higher finiteness properties of 
discrete groups is a very active field of enquiry, with generalisations of subgroups of products of free groups playing a central role, e.g. \cite{BesBra-97}, \cite{BriHowMilSho-09, BriHowMilSho-13}. In particular, it has been recognised that the finiteness properties of subgroups in direct products of surface groups (more generally, residually-free groups) play a dominant role in determining the structure of these subgroups  \cite{BriHowMilSho-09}. In parallel, it has been recognised, particularly following the work of Delzant and Gromov \cite{DelGro-05},  that subgroups of direct products of surface
groups play an important role in the investigation of K\"ahler groups (see also \cite{Py-13, DelPy-16}).

Given this context, it is natural that the first examples of K\"ahler groups with exotic finiteness properties
should have been constructed as the kernels of maps from a product of hyperbolic surface groups to an abelian group.
This breakthrough was achieved by Dimca, Papadima and Suciu \cite{DimPapSuc-09-II}. Further examples
were constructed by Biswas, Mj and Pancholi \cite{BisMjPan-14} and by Llosa Isenrich \cite{Llo-16-II}. Our main purpose here is to
construct examples of a different kind. 

A {\em Kodaira fibration} (also called a {\em regularly fibred surface}) is a compact complex surface $X$ that admits a
holomorphic submersion onto a smooth complex curve. Topologically, $X$ is the total space of a smooth fibre bundle
whose base and fibre are closed 2-manifolds (with restrictions on the holonomy).
These complex surfaces bear Kodaira's name because he \cite{Kod-67} (and independently Atiyah \cite{Ati-69}) constructed specific non-trivial examples in order to prove that the signature is
not multiplicative in smooth fibre bundles. Kodaira fibrations should not be confused with Kodaira surfaces in the sense of \cite[Sect. V.5]{BarHulPetVdV-04}, which are complex surfaces of Kodaira dimension zero that are never K\"ahler.

The new classes of K\"ahler groups that we shall
construct will appear as the fundamental groups of generic fibres
of certain holomorphic maps from a product of Kodaira fibrations to an elliptic curve. The first and
most interesting family arises from a detailed construction of complex surfaces of positive signature
that is adapted from Kodaira's original construction \cite{Kod-67}. In fact, our surfaces are
diffeomorphic to those of Kodaira but have a different complex structure.
The required control over the finiteness properties
of these examples comes from the second author's  work on products of branched covers of elliptic curves
\cite{Llo-16-II}, which in turn builds on \cite{DimPapSuc-09-II}.

In order to  obviate the concern that our groups might be disguised perturbations of known examples,
we prove that no subgroup of finite index can be embedded in a direct product of surface groups. We do this
by proving that any homomorphism from the subgroup to a residually-free group must have infinite
kernel (Section 6).

\begin{theorem} \label{thmNonCommIntro}
For each $r\geq 3$ there exist Kodaira fibrations $X_i$, $i=1,\dots, r$, and a holomorphic map from $X=X_1\times \cdots \times X_r$ onto an elliptic curve $E$, with generic fibre $\overline{H}$, such that the sequence
\[
 1 \rightarrow \pi_1 \overline{H}\rightarrow \pi_1 X\rightarrow \pi_1 E \rightarrow 1
\]
is exact and $\pi_1 \overline{H}$ is a K\"ahler group that is of type $\mathcal{F}_{r-1}$ but not $\mathcal{F}_r$. 

Moreover, no subgroup of finite index in $\pi_1\overline{H}$ embeds in a direct product of surface groups.
\end{theorem}

We also obtain K\"ahler groups with exotic finiteness properties
from Kodaira fibrations of signature zero. Here the constructions
are substantially easier and do not take us far from subdirect products of surface groups.
Indeed it is not difficult to see that all of the groups that arise in this setting have a subgroup of finite index
that embeds in a direct product
of surface groups; it is more subtle to determine when the groups themselves admit such an embedding --- 
this is almost 
equivalent to deciding which Kodaira fibrations have a fundamental group
that is {\em residually free}, a problem solved in Section 6. The key
criterion is that for a Kodaira fibration $S_{\gamma}\hookrightarrow X {\rightarrow} S_{g}$, the preimage in ${\rm{Aut}}(\pi_1 S_{\gamma})$ of the holonomy representation $\pi_1S_{g}\to{\rm{Out}}(\pi_1 S_{\gamma})$ should be torsion-free (see Theorem \ref{t:resfree}). Here $S_{g}$ denotes a closed orientable surface of genus $g$.

\begin{theorem}
 Fix $r\geq 3$ and for $i=1,\dots, r$ let $S _{\g_i}\hookrightarrow X_i \stackrel{k_i}{\rightarrow} S_{g_i}$ be a topological surface-by-surface bundle such that $X_i$ admits a complex structure and has signature zero. Assume that $\g_i,g_i\geq 2$. Let $X=X_1\times \dots \times X_r$. Let $E$ be an elliptic curve and let $\alpha_i: S_{g_i}\rightarrow E$ be branched coverings such that the map $\sum_{i=1}^r \alpha_i:S_{g_1}\times \cdots \times S_{g_r}\rightarrow E$ is surjective on $\pi_1$. 
 
 Then we can equip $X_i$ and $S_{g_i}$ with K\"ahler structures such that:
\begin{enumerate}
\item the maps $k_i$ and $\alpha_i$ are holomorphic;
\item the map $f:=\sum_{i=1}^r \alpha_i\circ k_i:X\rightarrow E$ has connected smooth generic fibre $\overline{H}\stackrel{j}{\hookrightarrow} X$;
\item the sequence $$1\rightarrow \pi_1 \overline{H} \stackrel{j_{\ast}}{\rightarrow} \pi_1X\stackrel{f_{\ast}}{\rightarrow} \pi_1 E\rightarrow 1$$ is exact;
\item the group $\pi_1 \overline{H}$ is K\"ahler and of type $\mathcal{F}_{r-1}$, but not $\mathcal{F}_r$;
\item $\pi_1 \overline{H}$ has a subgroup of finite index that embeds in a direct product of surface groups.
\end{enumerate}
 \label{thmSgn0Intro}
\end{theorem}

Fibrations of the sort described in Theorem \ref{thmSgn0Intro} have been discussed in the context of Beauville surfaces and, more generally, quotients of products of curves; see Catanese \cite{Cat-00}, also e.g. \cite[Theorem 4.1]{BauCatGruPig-12}, \cite{DedPer-12}. There are some similarities between that work and ours,
in particular around the use of fibre products to construct fibrations with finite holonomy, but the purpose of our work is very different.

This paper is organised as follows. In Section 2 we generalise a theorem of Dimca, Papadima and Suciu about singular fibrations over elliptic curves to larger classes of fibrations, weakening the assumptions on the singularities. In Section 3 we study Kodaira fibrations of signature zero and prove Theorem 1.2. In Section 4 we construct the family of complex surfaces that will be used in Section 5 to construct the new K\"ahler groups described in Theorem 1.1. Finally, in Section 6 we explore the conditions under which the groups we have constructed can be embedded in direct products of surface groups (and residually-free groups).

\begin{acknowledgements*}
We thank Mahan Mj for helpful conversations related to the contents of this paper and the anonymous referee for their careful reading and helpful comments.
\end{acknowledgements*}

\section{Exact sequences associated to fibrations over complex curves}
\label{sec:KodGenThm}
Dimca, Papadima and Suciu proved the following theorem and used it to construct the first examples of K\"ahler groups with arbitrary finiteness properties.

\begin{theorem}[\cite{DimPapSuc-09-II}, Theorem C]
Let $X$ be a compact complex manifold and let $Y$ be a closed Riemann surface of genus at least one. Let $h: X\rightarrow Y$ be a surjective holomorphic map with isolated singularities and connected fibres. Let $\widehat{h}: \widehat{X}\rightarrow \widetilde{Y}$ be the pull-back of $h$ under the universal cover $p:\widetilde{Y} \rightarrow Y$ and let $H$ be the smooth generic fibre of $\widehat{h}$ (and therefore of $h$).

Then the following hold:
\begin{enumerate}
 \item $\pi_i(\widehat{X},H)=0$ for $ i \leq \mathrm{dim}H$
 \item If $\mathrm{dim} H \geq 2$, then  $1\rightarrow \pi_1 H\rightarrow \pi_1 X\overset{h_*}\rightarrow \pi_1 Y\rightarrow 1$ is exact.
\end{enumerate}
\label{thmC'}
\end{theorem}

We shall need the following generalisation, which follows from Theorem \ref{thmC'}(2)
by a purely topological argument.

\begin{theorem}
\label{thm2}
Let $Y$ be a closed Riemann surface of positive genus and let $X$ be a compact K\"ahler manifold. Let $f:X\rightarrow Y$ be a surjective holomorphic map with connected generic (smooth) fibre $\overline{H}$.

If $f$ factors as
 \[
  \xymatrix{X \ar[r]^{g} \ar[dr]_{f}  & Z \ar[d] ^{h}\\ 
	    & Y \\}
 \]
where $g$ is a locally trivial holomorphic fibration and $h$ is a surjective holomorphic map with connected fibres of complex dimension $n\ge 2$ and isolated singularities, then the following sequence is exact 

\[  1 \rightarrow \pi_1 \overline{H}\rightarrow \pi_1 X \overset{f_*}\rightarrow \pi_1 Y\rightarrow 1.
\]
\end{theorem}

\begin{proof} By applying Theorem \ref{thmC'} to the map $h:Z\rightarrow Y$ we get a short exact sequence
\begin{equation}\label{eq1}
1\rightarrow \pi_1 H\rightarrow \pi_1 Z \rightarrow \pi_1 Y \rightarrow 1.
\end{equation}

Let $p\in Y$ be a regular value such that $H=h^{-1}(p)$, let $j:H\hookrightarrow Z$ be the (holomorphic)
inclusion map, let $F\subset X$ be the (smooth) fibre of $g: X\rightarrow Z$, and identify
$\overline{H} =f^{-1}(p)=g^{-1}(H)$. The long exact sequence in homotopy for the fibration 
  \[
  \xymatrix{ F \ar@{^{(}->}[r]   & \overline{H} \ar[d] \\ 
	    & H \\}
 \]
begins
\begin{equation}\label{eq2}
\cdots \rightarrow \pi_2 H\rightarrow \pi_1F\rightarrow \pi_1 \overline{H} \rightarrow \pi _1 H \rightarrow 1 (=\pi_0F)\rightarrow \cdots.
\end{equation}
Let $\widehat{Z}\rightarrow Z$ be the regular covering with Galois group $\ker h_*$, let 
$\widehat{h}:\widehat{Z}\rightarrow \widetilde{Y}$ be a lift of $h$ and, as in Theorem \ref{thmC'},
identify $H$ with a connected component of its preimage in $\widehat{Z}$.   

In the light of Theorem \ref{thmC'}(1), the long exact sequence in homotopy for the pair $(\widehat{Z},H)$ 
implies that
$\pi_iH\cong \pi_i\widehat{Z}$ for $i\leq \mathrm{dim}H-1 = n-1$ and that the natural map 
$\pi_nH\to \pi_n\widehat{Z}$ is surjective. In particular, $\pi_2H
\to \pi_2\widehat{Z} \overset{\cong}\to \pi_2Z$ is surjective; this map is denoted by $\eta$
in the following diagram.

In this diagram, the first column comes from (\ref{eq2}), the second column is part of the long exact sequence
in homotopy for the fibration $g:X\to Z$, and the bottom row comes from (\ref{eq1}). The 
naturality of the long exact sequence in homotopy 
assures us that the diagram is commutative. We must prove that the middle row yields the short exact sequence
in the statement of the theorem.

\[
 \xymatrix{
 & \pi_2 H\ar[r]^{\eta}\ar[d] & \pi_2 Z \ar[d]  & &\\ & \pi_1 F   \ar[r]^{=} \ar[d] & \pi_1 F \ar[d]^{\lambda} & & \\
 & \pi_1 \overline{H} \ar[r]^{\iota}\ar[d]  &\pi_1 X \ar[r]^{f_*}\ar[d] &\pi_1 Y\ar[r]\ar[d] & 1 \ar[d] \\
 1\ar[r]& \pi_1H \ar[r]^{\delta} \ar[d] &\pi_1Z \ar[r]^{h_{\ast}}\ar[d] & \pi_1Y \ar[r] & 1\\
 & 1\ar[r] & 1 & &\\ }
\]

We know that $\delta$ is injective and $\eta$ is surjective, so a simple diagram
chase  (an easy case of the 5-Lemma) implies that the map $\iota$ is injective.

A further (more involved) diagram chase
proves exactness at $\pi_1 X$, i.e., that $\mathrm{Im}(\iota)=\ker(f_*)$.
\end{proof}

We will also need the following proposition. Note that the hypothesis on $\pi_2 Z\to \pi_1F$  is automatically
satisfied if $\pi_1F$ does not contain a non-trivial normal abelian subgroup. This is the case, for example,
if $F$ is a direct product of hyperbolic surfaces.

\begin{proposition}
\label{prop1part2}
Under the assumptions of Theorem \ref{thm2}, if the map $\pi_2 Z\to \pi_1F$ associated to the fibration $g:X\to Z$ is trivial, then (2.2) reduces to a short exact sequence
\[
1\rightarrow \pi_1 F\rightarrow \pi_1 \overline{H}\rightarrow \pi_1 H\rightarrow 1.
\]
If, in addition, the fibre $F$ is aspherical, then $\pi_ i \overline{H} \cong \pi_i H \cong \pi_i X$ for $2\leq i \leq n -1$.
\end{proposition}

\begin{proof} The commutativity of the top square in the above diagram implies that $\pi_2H\to\pi_1F$
is trivial, so (\ref{eq2}) reduces to the desired sequence.

If the fibre $F$ is aspherical then naturality of long exact sequences of fibrations and Theorem \ref{thmC'}(1) imply that we obtain commutative squares
\[
\xymatrix{ \pi_ i \overline{H} \ar[r] \ar[d]^{\cong} & \pi_i  X \ar[d]^{\cong} \\
\pi_ i H \ar[r]^{\cong} & \pi_i Z }
\]
for $2\leq i \leq n-1$. It follows that $\pi_i \overline{H} \cong \pi_i H\cong  \pi_i X$ for $2\leq i \leq n-1$.
\end{proof} 

\section{Theorem \ref{thmSgn0Intro} and Kodaira fibrations of signature zero}
\label{sec:ExKod}

In this section we will prove Theorem \ref{thmSgn0Intro}. In order to 
explain the construction of the K\"ahler metrics implicit in the
statement, we need to first recall a construction of the second author \cite{Llo-16-II} that
provides the seed from which the failure of type $\FF_n$ in Theorem \ref{thmSgn0Intro} derives.

\vspace{.5cm}

\noindent{\bf Notation.} \textit{We write $\Sigma_g$ to denote the closed orientable surface of genus $g$.}

\subsection{The origin of the lack of finiteness}
 
The first examples of K\"ahler groups with exotic finiteness properties were constructed by
Dimca, Papadima and Suciu  in \cite{DimPapSuc-09-II} by considering a particular map from a 
product of hyperbolic surfaces to an elliptic curve. The following construction 
of the second author  \cite{Llo-16-II} extends their result to a much wider class of
maps.

Let $E$ be an elliptic curve, i.e. a 1-dimensional complex torus, and 
for $i=1,\dots,r$ let $h_i: \Sigma_{g_i}\rightarrow E$ be a branched cover, where each $g_i\geq 2$.
Endow $\Sigma_{g_i}$ with the complex structure that makes $h_i$ holomorphic. 
Let $Z=\Sigma_{g_1}\times \cdots \times \Sigma_{g_r}$.
Using the additive structure on $E$, we define a surjective map with isolated singularities
\[
h = \sum_{i=1}^r h_i: Z\rightarrow E.
\]
In this setting, we have the following criterion describing the finiteness properties of the generic fibre of $h$: 

\begin{theorem}[\cite{Llo-16-II}, Theorem 1.1] 
For each $r\geq 3$, if $h_*:\pi_1Z\to \pi_1E$ is surjective,   
then the generic fibre $H$ of $h$ is connected and its fundamental group $\pi_1 H$ is a projective (hence K\"ahler) group that is of type $\mathcal{F}_{r-1}$ but not of type $\mathcal{F}_r$. Furthermore, the sequence 
\[
1\rightarrow \pi_1 H \rightarrow \pi_1 Z\overset{h_*}\rightarrow \pi_1 E\rightarrow 1
\]
is exact.
\label{thmLlI1}
\end{theorem}

\subsection{Kodaira Fibrations} 

The following definition is equivalent to the more concise one that we gave in the introduction.

\begin{definition}
A \textit{Kodaira fibration} $X$ is a compact K\"ahler surface  (real dimension 4) that admits a 
holomorphic submersion $X\rightarrow \Sigma_g$.  The fibre of $X\rightarrow \Sigma_g$ will
be a closed surface, $\Sigma_\gamma$ say. Thus, topologically, $X$ is a $\Sigma_\gamma$-bundle over $\Sigma_g$.
We require $g,\gamma\ge 2$.
\end{definition}

The nature of the holonomy in a Kodaira fibration is intimately related to the \textit{signature $\sigma(X)$},
which is the signature of the bilinear form
\[
\cdot \cup \cdot : H^2(X,\RR)\times H^2(X,\RR)\rightarrow H^4(X,\RR)\cong \RR
\]
given by the cup product.

\subsection{Signature zero: groups commensurable to subgroups of direct products of surface groups}
\label{sec:sgn0}

We will make use of the following theorem of Kotschick \cite{Kot-99} and a detail from
his proof. Here, 
$\Mod(\Sigma_g)$ denotes the mapping class group of $\Sigma_g$.

\begin{theorem}
\label{thm3}
Let $X$ be a (topological) $\Sigma_\gamma$-bundle over $\Sigma_g$ where $g,\gamma\ge 2$. Then the following are equivalent:
\begin{enumerate}
\item $X$ can be equipped with a complex structure 
and $\sigma(X)=0$;
\item the monodromy representation $\rho: \pi_1\Sigma_g\rightarrow {\rm{Out}}(\pi_1\Sigma_\gamma)=\Mod(\Sigma_\gamma)$ has finite image.
\end{enumerate}
\end{theorem}

Note that condition (2) implies in particular that $X$ is isogenous to a product of complex curves.
The following is an immediate consequence of the proof of Theorem \ref{thm3} in \cite{Kot-99}.

\begin{addendum}
\label{corkot}
If either of the equivalent conditions in Theorem \ref{thm3} holds, then for any complex structure on 
the base space $\Sigma_g$
 there is a K\"ahler structure on $X$ with respect to which the projection $X\to\Sigma_g$ is holomorphic.
\end{addendum}

We are now in a position to construct the examples promised in Theorem \ref{thmSgn0Intro}.
Fix $r\ge 3$ and for $i=1,\cdots, r$ let $X_i$ be the underlying manifold
of a Kodaira fibration with base $\Sigma_{g_i}$ and fibre $\Sigma_{\gamma_i}$. Suppose that
$\sigma(X_i)=0$. Let $Z= \Sigma_{g_1}\times \cdots \times \Sigma_{g_r}$.

We fix an elliptic curve $E$ and
choose branched coverings $h_i:\Sigma_{g_i}\rightarrow E$ so that $h:=\sum_ih_i$ induces a
surjection $h_*:\pi_1Z\to \pi_1E$. We endow $\Sigma_{g_i}$ with the complex structure that
makes $h_i$ holomorphic and use Addendum \ref{corkot} to choose a complex structure on $X_i$ that makes $p_i:X_i\to\Sigma_{g_i}$ holomorphic. Let $X=X_1\times\dots X_r$ and let $p:X\to Z$ be the map that restricts to $p_i$
on $X_i$.

\begin{theorem} Let $p:X\to Z$ and $h:Z\to E$ be the maps defined above,
let $f=h\circ p: X\rightarrow E$ and let $\overline{H}$ be the generic smooth fibre of $f$. Then  $\pi_1\overline{H}$ is a K\"ahler group of type $\mathcal{F}_{r-1}$ that is not of type $\mathcal{F}_r$ and there is a short exact sequence
\[
1\rightarrow \pi_1 \overline{H}\rightarrow \pi_1 X\overset{f_*}\rightarrow \pi_1E=\ZZ^2\rightarrow 1.
\]
Moreover, $\pi_1\overline{H}$ has a subgroup of finite index that embeds in a direct product of surface groups.
\label{thm4}
\end{theorem}

We shall need the following well known fact.

\begin{lemma}
\label{propbieri}
 Let $N$ be a group with a finite classifying space and assume that there is a short exact sequence
 \[
 1\rightarrow N\rightarrow G\rightarrow Q\rightarrow 1.
 \]
Then $G$ is of type $\FF_n$ if and only if $Q$ is of type $\FF_n$.
\end{lemma}
\begin{proof}
See \cite[Proposition 2.7]{Bie-81}.
\end{proof}

\begin{proof}[Proof of Theorem \ref{thm4}]
By construction, the map $f=p\circ h: X\rightarrow E$ satisfies the hypotheses of Theorem \ref{thm2}.
Moreover, since $Z$ is aspherical, $\pi_2 Z=0$ and Proposition \ref{prop1part2} applies. Thus, writing
$\overline{H}$ for the generic smooth fibre of $f$ and $H$ for the generic smooth fibre of $h$, we have short exact sequences 
\[
1\rightarrow \pi_1 \overline{H}\rightarrow \pi_1 X\rightarrow \pi_1 E =\ZZ^2\rightarrow 1
\]
and
\[
1\rightarrow \pi_1 \Sigma_{\gamma_1}\times \cdots \times \pi_1\Sigma_{\gamma_r}\rightarrow \pi_1 \overline{H}\rightarrow \pi_1H\rightarrow 1.
\]
The product of the closed surfaces $\Sigma_{\gamma_i}$ is a classifying space for the kernel in
the second sequence, so Lemma \ref{propbieri} implies that $\pi_1\overline{H}$ is of type $\FF_k$ if and only if $\pi_1H$ is of type $\FF_k$.  Theorem \ref{thmLlI1} tells us that $\pi_1 H$ is of type $\mathcal{F}_{r-1}$ and not of type $\mathcal{F}_r$.

To see that $\pi_1 \overline{H}$ is commensurable to a subgroup of a direct product of surface groups,
note that the assumption $\sigma(X_i)=0$ implies that the monodromy representation 
$\rho_i:\pi_1\Sigma_{g_i}\to {\rm{Out}}(\pi_1\Sigma_{\gamma_i})$ is finite, and hence $\pi_1X_i$ contains the product of
surface groups $\Gamma_i=\pi_1\Sigma_{\gamma_i}\times \ker \rho_i$ as a subgroup of finite index. (Here we
are using the fact that the centre of $\Sigma_{\gamma_i}$ is trivial -- cf. Corollary 8.IV.6.8 in \cite{Bro-82}).
The required subgroup of finite index in  $\pi_1 \overline{H}$ is its intersection
with $\Gamma_1\times\dots\times\Gamma_r$.
\end{proof}

In the light of Theorem \ref{thm4}, all that remains unproved in Theorem \ref{thmSgn0Intro} is the
assertion that in general $\pi_1 \overline{H}$ is not itself a subgroup of a product of surface groups.
We shall return to this point in the last section of the paper.

\section{New Kodaira Fibrations $X_{N,m}$}
\label{sec:NewKod}

In 1967 Kodaira \cite{Kod-67} constructed a family of complex surfaces $M_{N,m}$ that fibre over a
complex curve but have positive signature.
(See \cite{Ati-69} for a very similar construction by Atiyah, and \cite{BD} for a more recent variation.)
We shall produce a new
family of K\"ahler surfaces $X_{N,m}$ that are Kodaira fibrations. 
We do so by adapting Kodaira's construction in a manner 
designed to allow appeals to Theorems \ref{thm2} and \ref{thmLlI1}. This is the main innovation in our construction of
new families of K\"ahler groups. 
 
Our surface $X_{N,m}$ is diffeomorphic to Kodaira's
surface $M_{N-1,m}$ but it has a different complex structure.  
Because signature is a topological invariant,
we can appeal to Kodaira's calculation of the signature
\begin{equation}
 \sigma(X_{N,m})= 8m^{4N}\cdot N\cdot m\cdot (m^2-1)/3.
 \label{eqnSign}
\end{equation}
The crucial point for us is that
$\sigma(X_{N,m})$ is non-zero:  Theorem \ref{thm3} implies 
that the monodromy representation
associated to the Kodaira fibration $X_{N,m}\to \Sigma$ has infinite image, from which it follows that
the K\"ahler groups with exotic finiteness properties constructed in Theorem \ref{thmKodExNotComm} are
not commensurable to subgroups of direct products of surface groups,
as we shall see in Section 6.

\begin{remark}
{\em{A priori}}, Kodaira's construction depends on choices of generating set for the fundamental groups of certain punctured Riemann surfaces. It is not clear to us whether  different choices could lead to non-homeomorphic surfaces. When we say that our surfaces $X_{N,m}$ are diffeomorphic (or homeomorphic) to Kodaira's surfaces $M_{N-1,m}$ this is under the assumption that we have made the same choices of generators as him.
The comparison is valid for all possible choices, i.e. the class of all differentiable 4-manifolds one can obtain from our adaption of Kodaira's construction is isomorphic to the class of all differentiable 4-manifolds Kodaira obtains from his construction. 
\end{remark}

\subsection{The construction of $X_{N,m}$}

Kodaira's  construction of $M_{N,m}$ begins with a regular finite-sheeted covering of a higher
genus curve $S\to R$. He branches $R\times S$ along the union of two curves: one is the graph
of the covering map and the other is the graph of the covering map twisted by a certain involution. We shall
follow this template, but rather than beginning with a regular covering, we begin with a carefully crafted
branched covering of an elliptic curve; this is a crucial feature, as it allows us to apply Theorems \ref{thm2} and
\ref{thmLlI1}. Our covering is designed to admit an involution that allows us to follow the remainder of 
Kodaira's argument.

\begin{figure}[ht] 
 \includegraphics[width=13cm,height=19.5cm,keepaspectratio]{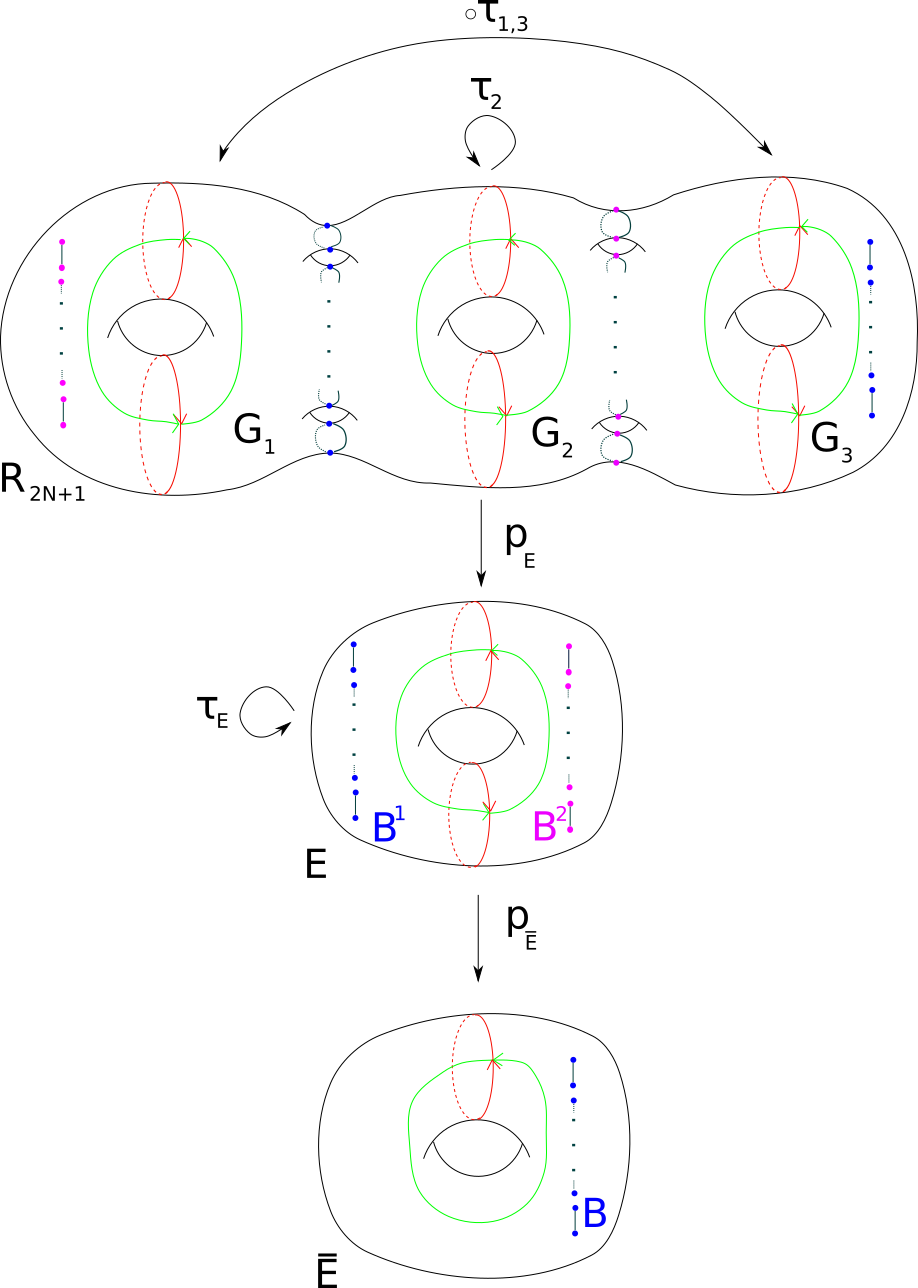}
 \caption{$R_{2N+1}$ as a branched covering of $E$, and the involution $\tau_E$}
 \vspace{-0.13cm}
 \label{fig:Kodaira1}
\end{figure}

Let $\overline{E}=\C / \Lambda$ be an elliptic curve. Choose a finite set of (branching) points $B=\left\{b_1,\cdots,b_{2N}\right\}\subset \overline{E}$ and fix a basis $\overline{\mu}_1,\overline{\mu}_2$ of $\Lambda\cong \pi_1\overline{E}\cong \ZZ ^2$ represented by simple loops $\mu_i^0$ in $\overline{E}\setminus B$. Let $p_{\overline{E}}:E\rightarrow \overline{E}$ be the double covering that the Galois correspondence associates to the homomorphism
$\Lambda \to \ZZ_2$ killing $\overline{\mu}_1$. Let $\mu_1$ be the preimage in $E$ of $\mu_1^0$ (it has two components)
 and let $\mu_2$ be the unique lift of $2\cdot\overline{\mu_2}$. Note that $\pi_1 E$ is generated by $\mu_2$ and either component of $\mu_1$. 
 
$E$ has a canonical complex structure making the covering map holomorphic; with this
structure, $E$ is an elliptic curve.

Let $\tau_E:E\rightarrow E$ be the generator of the Galois group of $E\to\overline{E}$; it is holomorphic and 
interchanges the components of $E\smallsetminus \mu_1$.

Denote by $B^{(1)}$ and $B^{(2)}$ 
the preimages of $B$ in the two distinct connected components of $E\smallsetminus \mu_1$.
The action of $\tau_E$ interchanges these sets.

Choose pairs of points in $\left\{b_{2k-1},b_{2k}\right\}\subset B$, $k=1,\cdots,N$, connect them by disjoint arcs $\gamma_1,\cdots,\gamma_{N}$ that do not intersect
$\mu_1^0$ and lift these arcs to $E$. Denote by $\gamma _1^1,\cdots, \gamma_N^1$ the arcs joining points in $B^{(1)}$ and by $\gamma_1^2,\cdots, \gamma_N^2$ the arcs joining points in $B^{(2)}$. 

Next we define a 3-fold branched covering of $E$ as follows. Take three copies $F_1$, $F_2$ and $F_3$ of $E\setminus (B^{(1)}\cup B^{(2)})$ identified with $E\setminus (B^{(1)}\cup B^{(2)})$ via maps $j_1$, $j_2$ and $j_3$. We obtain surfaces $G_1$, $G_2$ and $G_3$ with boundary by cutting $F_1$ along the arcs $\gamma^2_1,\cdots,\gamma^2_N$, cutting $F_2$ along the arcs $\gamma^i_k$, $i=1,2$, $k=1,\cdots, N$ and cutting $F_3$ along the arcs $\gamma ^1_1,\cdots, \gamma^1_N$. Identify the two copies of the open arc $\gamma_k^1$ in $F_2$ with the two copies of the open arc $\gamma_k^1$ in $F_3$ and identify the two copies of the open arc $\gamma_k^2$ in $F_2$ with the two copies of the open arc $\gamma_k^2$ in $F_1$ in the unique way that makes the map $p_E: G_1\cup G_2 \cup G_3 \mapsto E\setminus (B^{(1)}\cup B^{(2)})$ induced by the identifications $j_i$ of $F_i$ with $E\setminus (B^{(1)}\cup B^{(2)})$ a covering map; the map $p_E$ is portrayed in Figure \ref{fig:Kodaira1}.

Consider the closed surface $R_{2N+1}$ of genus $2N+1$ obtained by closing the cuts
and punctures of $G_1\cup G_2\cup G_3$. It is clear that the map $p_E$ extends to a 3-fold  covering map $R_{2N+1}\to E$. We continue to denote this extended map by $p_E$. 
There is a unique complex structure on $R_{2N+1}$ making $p_E:R_{2N+1}\to E$ holomorphic.

The map $\tau_E$ induces a continuous involution $\tau_2: G_2\rightarrow G_2$ and a continuous involution $\tau_{1,3}:G_1 \sqcup G_3\rightarrow G_1 \sqcup G_3$ without fixed points: these are defined by requiring the following
diagrams to commute 
\[
 \xymatrix{ G_2\ar[r]^{\tau_2} \ar[d]_{j_2} & G_2 \ar[d]_{j_2}\\
 E\setminus (B^{(1)}\cup B^{(2)}) \ar[r]^{\tau_E} & E\setminus (B^{(1)} \cup B^{(2)})\\}
\]
\[
 \xymatrix{ G_1\ar[r]^{\tau_{1,3}}\ar[d]_{j_1} & G_3\ar[d]_{j_3}\\
 E\setminus (B^{(1)}\cup B^{(2)}) \ar[r]^{\tau_E} & E\setminus (B^{(1)} \cup B^{(2)})\\}
\]
\[
 \xymatrix{ G_3\ar[r]^{\tau_{1,3}}\ar[d]_{j_3} & G_1\ar[d]_{j_1}\\
 E\setminus (B^{(1)}\cup B^{(2)}) \ar[r]^{\tau_E} & E\setminus (B^{(1)} \cup B^{(2)})\\}
\]
where $j_i$ denotes the unique continuous extension of the original identification $j_i:F_i\to E\smallsetminus
(B^{(1)}\cup B^{(2)})$.

The maps $\tau_2$ and $\tau_{1,3}$ coincide on the intersection of $G_1\sqcup G_3$ with $G_2$ and together they define a continuous involution 
$\tau_R':R_{2N+1}\setminus p_E^{-1}(B^{(1)}\cup B^{(2)})\rightarrow R_{2N+1}\setminus p_E^{-1}(B^{(1)}\cup B^{(2)})$ which extends to a continuous involution 
$\tau_R: R_{2N+1}\rightarrow R_{2N+1}$.

Consider the commutative diagram
\[
 \xymatrix{R_{2N+1}\setminus p_E^{-1}(B^{(1)}\cup B^{(2)})\ar[r]^{\tau_R'} \ar[d]_{p_E} & R_{2N+1}\setminus p_E^{-1}(B^{(1)}\cup B^{(2)})\ar[d]_{p_E}\\
 E\setminus (B^{(1)}\cup B^{(2)}) \ar[r]^{\tau_E} & E\setminus (B^{(1)}\cup B^{(2)})\\}
\]
As $p_E$ is an unramified holomorphic  covering onto $E\setminus (B^{(1)}\cup B^{(2)})$ and $\tau_E$ is a holomorphic deck transformation  mapping $E\setminus (B^{(1)}\cup B^{(2)})$ onto itself, we can locally express $\tau _R'$ as the composition of holomorphic maps $p_E^{-1}\circ \tau _E \circ p_E$ and therefore $\tau _R'$ is itself holomorphic.

Since $\tau_R$ extends continuously to $R_{2N+1}$, it is holomorphic on $R_{2N+1}$, by Riemann's Theorem on removable singularities. By definition $\tau_R\circ \tau_R =\mathrm{Id}$. Thus $\tau_R: R_{2N+1}\rightarrow R_{2N+1}$ defines a holomorphic involution of $R_{2N+1}$ without fixed points. 

We have now manoeuvred ourselves into a situation whereby we can mimic Kodaira's construction.
We replace the surface $R$ in Kodaira's construction \cite[p.207-208]{Kod-67} by $R_{2N+1}$ and the involution $\tau$ in Kodaira's construction by the involution $\tau _R$. The adaptation is straightforward, but we shall recall the argument
below for the reader's convenience.

{\em{The result of this construction will be a new complex surface that we denote $X_{N,m}$.}} Arguing as in the proof
of \cite[Proposition 1]{Kot-99}, we see that $X_{N,m}$ is K\"ahler.

\subsection{Completing the Kodaira construction}

Let $\alpha_1,\beta_1,\cdots, \alpha_{2N+1},\beta_{2N+1}$ denote a standard set of generators of $\pi_1 R_{2N+1}$ satisfying the relation $\left[\alpha_1,\beta_1\right]\cdots\left[\alpha_{2N+1},\beta_{2N+1}\right]=1$,
chosen so that  $\alpha_1, \alpha_2$ and $\alpha_3$ correspond to the preimages of 
$\mu_1$ in $G_1$, $G_2$ and $G_3$, and $\beta_1, \beta_2, \beta_3$ correspond to the preimages of $\mu_2$  (with tails connecting these loops to a common base point).

For $m\in \ZZ$ consider the $m^{2(2N+1)}$-fold covering $q_R:S\rightarrow R_{2N+1}$ corresponding to the homomorphism 
\begin{equation}
\begin{split}
 \pi_1 R_{2N+1} &\rightarrow (\ZZ/m\ZZ)^{2(2N+1)}\\
 \alpha_i &\mapsto (0,\cdots,0,1_{2i-1},0,0,\cdots,0)\\
 \beta_i & \mapsto (0,\cdots,0,0,1_{2i},0,\cdots,0),
 \end{split}
 \label{eqnCovhom}
\end{equation}
where $1_{i}$ is the generator in the $i$-th factor. By multiplicativity of the Euler characteristic, we see that the genus of $S$ is $2N\cdot m^{2(2N+1)}+ 1$.

To simplify notation we will from now on omit the index $R$ in $q_R$ and $\tau_R$, as well as the index $2N+1$ in $R_{2N+1}$, and we denote the image $\tau(r)$ of a point $r\in R$ by $r^*$. Let $q^*=\tau\circ q:S\rightarrow R$, let $W=R\times S$ and let
\[
 \Gamma = \left\{(q(u),u)\mid u\in S\right\},
\]
\[
 \Gamma ^* = \left\{ (q^*(u),u)\mid u\in S\right\}
\]
be the graphs of the holomorphic maps $q$ and $q^*$. Let $W''=W\setminus(\Gamma\cup \Gamma^*)$. {\em{We shall define the complex surface $X_{N,m}$}}
 as an $m$-fold branched covering of $W$ branched along $\Gamma$ and $\Gamma^*$. Its construction makes use of the following Lemma from \cite[p.209]{Kod-67}:
\begin{lemma}
 Fix a point $u_0\in S$, identify $R$ with $R\times u_0$ and let $D$ be a small disk around $t_0=q(u_0)\in R$. Denote by $\gamma$ the positively oriented boundary circle of $D$. Then $\gamma$ generates a cyclic subgroup $\langle \gamma\rangle$ of order $m$ in $H_1(W'', \ZZ)$ and
 \begin{equation}
 \label{eqn:KodairaHomology}
  H_1(W'',\ZZ)\cong H_1(R,\ZZ)\oplus H_1(S,\ZZ)\oplus \langle \gamma \rangle.
 \end{equation}
 \label{lemKodSplitting}
\end{lemma}

The proof of this lemma is purely topological and in particular makes no use of the complex structure on $W''$. From a topological point of view our manifolds and maps are equivalent to Kodaira's manifolds and maps, i.e.
there is a  homeomorphism between the $W''$ in our work and the $W''$ in Kodaira's work that makes all of the obvious diagrams commute. However, there is a subtle point that Kodaira does not seem to address which is that the isomorphism in \eqref{eqn:KodairaHomology} depends on a choice of splitting of the epimorphism $H_1(W'',\ZZ)\to H_1(R,\ZZ)\oplus H_1(S,\ZZ)$.

We impute that the following splitting is implicit in Kodaira's work: choose\footnote{The
topology of the covering space $X''$ that we will construct, and hence that of $X_{N,m}$, might depend on this choice, but all choices lead to Kodaira fibrations to which the
remainder of Kodaira's argument applies.}
 simple closed loops $\zeta_1,\dots, \zeta_{2(2N+1)}: [0,1]\to R\setminus \left\{t_0 \cup t_0^{\ast}\right\}$ representing a symplectic generating set for $H_1(R,\ZZ)$ and simply closed representatives  $\xi _1 ,\dots, \xi_{2 \cdot g(S)}:\left[0,1\right]\to S \setminus q_R^{-1}(\left\{t_0\cup t_0^{\ast}\right\})$ of a symplectic generating set for $H_1(S,\ZZ)$ in such a way that all conditions in the proof of Lemma \ref{lemKodSplitting} on these generators are satisfied. It is clear from Kodaira's proof of Lemma \ref{lemKodSplitting} that such a choice of generators exists and it 
follows easily from his proof  that the composition of these representatives with the canonical inclusions $R\setminus \left\{t_0 \cup t_0^{\ast}\right\} \hookrightarrow W''$ and $S \setminus q_R^{-1}(\left\{t_0\cup t_0^{\ast}\right\})$ defines a splitting $H_1(R,\ZZ) \oplus H_1(S,\ZZ) \hookrightarrow H_1(W'',\ZZ)$. The cokernel of this map gives us
the desired epimorphism  $H_1(W'',\ZZ) \to \langle \gamma \rangle$. The composition of this epimorphism with the abelianization map $\pi_1W''\rightarrow H_1(W'',\ZZ)$ induces an epimorphism $\kappa: \pi_1W''\rightarrow \langle\gamma\rangle$. Consider the $m$-sheeted covering $X''\rightarrow W''$ corresponding to the kernel of this map and equip $X''$ with the complex structure that makes the covering map holomorphic.

We claim that the covering $X''\to W''$ extends to an $m$-fold ramified covering $X_{N,m}\to W$ with branching locus $\Gamma\cup\Gamma ^*$, where $X_{N,m}$ is a closed
complex surface and 
the restriction to the preimage of $\Gamma \cup \Gamma^{\ast}$ is biholomorphic. For this,
it suffices to check two things in a neighbourhood of $\Gamma$ (with
entirely similar arguments applying to $\Gamma^{\ast}$):
\begin{enumerate}
 \item Let $N(\Gamma)$ be a tubular neighbourhood of $\Gamma$ in $W$,
 identified with a neighbourhood of the zero-section in the normal bundle. Then 
 $\partial N(\Gamma)$ is a circle bundle over $\Gamma$.
  Let $j: \partial N(\Gamma) \hookrightarrow W''$ be the inclusion. 
  Then $\kappa\circ j_*^{-1}$ maps the fundamental group of the circle fibre onto $\langle \gamma \rangle$ and the canonical epimorphism $\pi_1 \partial N(\Gamma)\to \pi_1 \Gamma$ maps $j_*^{-1}(\ker \kappa)$ onto $\pi_1 \Gamma$. Thus the preimage of $\partial N(\Gamma)\subset W''$ in $X''$ is the total space of 
a connected circle bundle over $\G$ and, topologically, $X''\to W''$ extends to a ramified covering $p:X_{N,m}\to R \times S$.
 \item With (1) established, we then need to check that the complex structure on $X''$ can be extended to a complex structure on $X_{N,m}$ with respect to which the topological
 covering is holomorphic.
\end{enumerate}

Our construction of $\kappa$ is crafted to make (1) obvious. 
For property (2) we need a complex structure on the charts around
points of $\Gamma\subset X_{N,m}$ making $p:X_{N,m}\to R\times S$ 
holomorphic near these points. The Riemann Extension Theorem then assures
the existence of the desired global complex structure on $X_{N,m}$. 

On each neighbourhood of a point on $\Gamma\subset R\times S$ 
we want to choose complex coordinates so that the covering is given locally
by $(z,w)\to (z^m,w)$, and $\Gamma$ is $\left\{z=0\right\}$. To this end,
if $(u,v)$ are the coordinates on a product neighbourhood $U\times V\subset R\times S$ of a point in $\Gamma$, we define new coordinates on $U\times V$ by $(z(u,v),w(u,v))= (u-q_R(v),v)$. Looking at the corestriction of $X''\to W''$ to the level sets $w\equiv w_0 =\mathrm{const}$, we see that in these coordinates the holomorphic map is given by an $m$-fold branched covering $(z,w_0)\to (z^n,w_0)$ of Riemann surfaces, as desired. 

The same considerations apply in a neighbourhood of $\Gamma^{\ast}$, so the construction
of the complex surface $X_{N,m}$ is complete.
\medskip

The composition of the covering map $X_{N,m}\rightarrow W$ and the projection $W=R\times S\rightarrow S$ induces a holomorphic submersion $\psi:X_{N,m}\rightarrow S$ with complex fibre $R'=\psi^{-1}(u)$ a closed Riemann surface that is an $m$-sheeted branched covering of $R$ with branching points $q(u)$ and $q^*(u)$ of order $m$. The complex structure of the fibres varies: each pair of fibres is homeomorphic
but not (in general) biholomorphic.

\section{Construction of K\"ahler groups}

We fix an integer $m\geq 2$ and associate to each $r$-tuple of 
positive integers $\nn=(N_1,\cdots, N_r)$ with $r\geq 3$ the product of the complex surfaces $X_{N_i,m}$ constructed in the previous section:
$$
X(\nn, m) = X_{N_1,m} \times\dots\times X_{N_r,m}.
$$
Each $X_{N_i,m}$ was constructed to have a holomorphic projection $\psi_i:X_{N_i,m}\to S_i$ with
fibre $R_i'$.  

By construction, each of the Riemann surfaces $S_i$ comes with
a holomorphic map $f_i=p_i\circ q_i$, where $p_i=p_{E,i}:R_{2N_i+1}\rightarrow E$ and $q_i=q_{R,i}:S_i\rightarrow R_{2N_i+1}$. We also need the homomorphism defined in (\ref{eqnCovhom}),
which we denote by $\theta_i$.

We want to determine what $f_{i*}(\pi_1S_i)\trianglelefteq \pi_1E$ is. By definition $q_{i*}(\pi_1S_i)=\mathrm{ker}(\theta_i)$, so $f_{i*}(\pi_1S_i)=p_{i*}(\mathrm{ker}\theta_i)$. The map $\theta _i$ factors through the abelianization $H_1(R_i,\ZZ)$ of $\pi_1 R_i$, yielding $\overline{\theta} _i: H_1(R_i,\ZZ)\rightarrow (\ZZ/m\ZZ)^{2(2N_i+1)}$,
which has the same image in $H_1(E,\ZZ) = \pi_1E$ as $f_{i*}(\pi_1S_i)$.

Now,
\[
\mathrm{ker}\overline{\theta}_i = \langle m\cdot\left[\alpha_1\right],m\cdot\left[\alpha_1\right],m\cdot\left[\beta_1\right],\cdots,m\cdot\left[\alpha_{2N_i+1}\right],m\cdot\left[\beta_{2N_i+1}\right]\rangle \leq H_1(R_i,\ZZ).
\]
and $\alpha_j$, $\beta_j$ were chosen such that for $1\leq i \leq r$ we have
\[
 p_{i*}\left[\alpha_j\right]=\left\{ \begin{array}{ll} \mu_1 &\mbox{, if } j\in\left\{1,2,3\right\}\\0 &\mbox{, else }\\ \end{array} \right.  
 \mbox{   and    } 
 p_{i*}\left[\beta_j\right]=\left\{\begin{array}{ll} \mu_2 &\mbox{, if } j\in\left\{1,2,3\right\}\\0 &\mbox{, else }\\\end{array}\right. .
\]
(Here we have abused notation to the extent of writing $\mu_1$ for the unique element of $\pi_1E=H_1E$ determined
by either component of the preimage of $\overline{\mu}_1$ in $E$.) Thus,  
\begin{equation}
\label{eqn:pi1E}
 f_{i*}(\pi_1 S_i) =\langle m\cdot \mu_1,m\cdot \mu_2\rangle\leq \pi_1 E.
\end{equation}
There are three {\em loops} that are lifts $\mu_{1,i}^{(1)},\mu_{1,i}^{(2)},\mu_{1,i}^{(3)}$ of $\mu_1$ with respect to $p_i$ (regardless of the choice of basepoint $\mu_{1,i}^{(j)}(0)\in p_i^{-1}(\mu_1(0))$). The same holds for $\mu_2$. And by choice of $\alpha_j,\beta_j$ for $j\in\left\{1,2,3\right\}$, we have $\left[\mu_1^{(j)}\right]=\left[\alpha_j\right]\in H_1(R_i,\ZZ)$ after a permutation of indices.

Denote by $q_E:E'\rightarrow E$ the $m^2$-sheeted covering of $E$ corresponding to the subgroups $f_{i*}(\pi_1S_i)$. Endow $E'$ with the unique complex structure making $q_E$ holomorphic. By \eqref{eqn:pi1E} the covering and the complex structure are independent of $i$.

Since $f_{i*}(\pi_1S_i)= q_{E*}(\pi_1E')$ there is an induced surjective map $f'_i:S_i\rightarrow E'$ making the diagram
\begin{equation}
\label{diag:Ex3sq1}
 \xymatrix{S_i\ar[d]_{f'_i}\ar[r]^{q_i} \ar[rd]^{f_i} & R_i\ar[d]^{p_i}\\ E'\ar[r]^{q_E} & E\\}
\end{equation}
commutative. The map $f'_i$ is surjective and holomorphic, since $f_i$ is surjective and holomorphic and $q_E$ is a holomorphic covering map.

\begin{lemma}
 Let $B'= q _{E} ^{-1} (B)$, $B_{S_i} = f_{i} ^{-1} (B) = f_{i} ^{'-1} (B')$. Let $\mu'_1,\mu'_2:[0,1]\to
 E'\setminus B'$ be loops that generate $\pi_1 E'$ 
 and are such that $q_E \circ \mu_1'=\mu_1^m$, $q_E \circ \mu_2'=\mu_2^m$.
 
 Then the restriction $f_i':S_i\setminus B_{S_i}\rightarrow E'\setminus B'$ is an unramified finite-sheeted covering map and all lifts of $\mu_1'$ and $\mu_2'$ with respect to $f_i'$ are loops in $S_i\setminus B_{S_i}$.
 \label{lem:Ex3Conn1}
\end{lemma}

\begin{proof}
 Since $f_i$ and $q_E$ are unramified coverings over $E\setminus B$, it follows from the commutativity of diagram \eqref{diag:Ex3sq1} that the restriction $f_i':S_i\setminus B_{S_i}\rightarrow E'\setminus B'$ is an unramified finite-sheeted covering map.
 
 For the second part of the statement it suffices to consider $\mu_1'$, since the proof of the statement for $\mu_2'$ is completely analogous. Let $y_0=\mu_1'(0)$, let $x_0\in f'^{-1}(y_0)$ and let $\nu_1: \left[0,1\right] \rightarrow S_i\setminus B_{S_i}$ be the unique lift of $\mu_1'$ with respect to $f_i'$ with $\nu_1(0)=x_0$. 
 
 Since $q_i$ is a covering map it suffices to prove that $q_i \circ \nu_1$ is a loop in $R_i$ based at $z_0=q_i(x_0)$ such that its unique lift based at $x_0$ with respect to $q_i$ is a loop in $S_i$.
 
By the commutatitivity of diagram (\ref{diag:Ex3sq1}) and the definition of $\mu_1'$, 
 \[
  \mu_1^m=q_E\circ \mu_1'=q_E\circ f_i'\circ \nu_1 = p_i\circ q_i \circ \nu_1.
 \]
 But the unique lift of $\mu_1^m$ starting at $z_0$ is given by $(\mu_{j_0}^1)^m$ where $j_0\in \left\{1,2,3\right\}$ is uniquely determined by $\mu_{j_0}^{(1)}(0)=z_0$. Uniqueness of path-lifting gives
 \[
  q_i\circ \nu _1 = (\mu_{j_0}^{(1)})^m.
 \]
Thus $(\mu_{j_0}^{(1)})^m\in \ker\theta_i= f_{i*}(\pi_1S_i)$. Now, $\mathrm{ker}\theta_i$ is normal
in  $\pi_1 R_i$ and $q_i:S_i\rightarrow R_i$ is an unramified covering map, so all lifts of $(\mu_{j_0}^{(1)})^m$ to $S_i$ are loops. In particular $\nu_1$ is a loop in $S_i$. 
\end{proof}

\begin{definition} A branched covering $\alpha: S\rightarrow T^2$ of a 2-torus $T^2$ with finite branch locus $B\subset T^2$ is \textit{purely-branched} if there are simple closed loops $\eta_1,\eta_2$ in $T^2\setminus B$, intersecting only in $\eta_1(0)=\eta_2(0)$, that generate $\pi_1 T^2$ and are such that the normal closure of $\left\{\eta_1,\eta_2\right\}$ in $\pi_1 \left(T^2\setminus B\right)$ satisfies $\left\langle \!\left\langle \eta_1, \eta _2 \right\rangle\! \right\rangle \leq \alpha_* (\pi_1 (S \setminus \alpha^{-1}(B)))$. 
\end{definition}

Lemma \ref{lem:Ex3Conn1} and the comment after \cite[Definition 2.2]{Llo-16-II} imply

\begin{corollary}
The holomorphic maps $f_i': S_i \rightarrow E'$ are purely-branched covering maps for $1\leq i \leq r$. In particular, the maps $f'_i$ induce surjective maps on fundamental groups.
 \label{cor:Ex3Cond3}
\end{corollary}

\begin{remark}
The second author of this paper introduced invariants for the K\"ahler groups arising in Theorem \ref{thmLlI1} and showed that these invariants lead to a complete classification of these groups in the special case where all the coverings are purely-branched. Thus Corollary \ref{cor:Ex3Cond3} ought to help in classifying the groups that arise from our construction. We shall return to this point elsewhere.
\end{remark} 
 
 Let $$Z_{\nn,m}= S_1\times \cdots\times S_r.$$ Using the additive structure on the elliptic curve $E'$ we combine the maps
 $f_i':S_i\to E'$ to define $h':Z_{\nn,m}\to E'$ by
  \[
  h': (x_1,\cdots,x_r) \mapsto \sum_{i=1}^r f_i'(x_i). 
 \]
 
 \begin{lemma}
  For all $m\ge 2$, all $r\ge 3$ and
  all $\nn=(N_1,\dots,N_r)$,
  the map $h':Z_{\nn,m}\to E'$ has isolated singularities and connected fibres.
  \label{lem:Ex3IsolConn}
 \end{lemma}
 \begin{proof}
  By construction, $f_i'$ is non-singular on $S_i\setminus B_{S_i}$ and $B_{S_i}$ is a finite set. 
  Therefore, the set of singular points of $h'$ is contained in the finite set
  \[
   B_{S_1}\times \cdots \times B_{S_r}.
  \]
  In particular, $h'$ has isolated singularities.
  
 Corollary \ref{cor:Ex3Cond3} implies that the $f'_i$ induce surjective maps on fundamental groups, so we can apply Theorem \ref{thmLlI1} to conclude that $h'$ has indeed connected fibres.
 \end{proof}

Finally, we define $g:X_{\nn,m}\to Z_{\nn,m}$ to be the product of the fibrations $\psi_i:X_{N_i,m}\to S_i$
and we define 
 \[
  f=h'\circ g: X_{\nn,m}\to E'.
 \]
Note that $g$ is a smooth fibration with fibre $F_{\nn,m}:=R_1'\times\dots\times R_r'$.  

With this notation established, we are now able to prove:
 
 \begin{theorem}
  Let $f: X_{\nn,m}\to E'$ be as above, let $\overline{H}_{\nn,m}\subset X_{\nn,m}$
   be the generic smooth fibre of $f$, and
  let $H_{\nn,m}$ be its image in $Z_{\nn,m}$.
  Then:
  \begin{enumerate}
  \item  $\pi_1 \overline{H}_{\nn,m}$ is a K\"ahler group that is of type $\mathcal{F}_{r-1}$ but not of type $\mathcal{F}_r$;
  \item  there are short exact sequences
  \[
   1\rightarrow \pi_1 F_{\nn,m}\rightarrow \pi_1\overline{H}_{\nn,m}\overset{g_*}\rightarrow \pi_1 H_{\nn,m}\rightarrow 1
  \]
  and
  \[
   1\rightarrow \pi_1 \overline{H}_{\nn,m}\rightarrow \pi_1 X_{\nn,m}\overset{f_*}\rightarrow \Z^2 \rightarrow 1,
  \]
  such that the monodromy representation $\pi_1H_{\nn,m}\to{\rm{Out}}(\pi_1F_{\nn,m})$ has infinite image;
  \item no subgroup of finite index in $\pi_1\overline{H}_{\nn,m}$ embeds in a direct product of surface groups (or of residually free groups).
  \end{enumerate}
  \label{thmKodExNotComm}
 \end{theorem}
 
 \begin{proof} We have constructed $\overline{H}_{\nn,m}$ as the fundamental group of a K\"ahler manifold,
 so the first assertion in (1) is clear.
 
 We argued above that all of
 the assumptions of Theorem \ref{thm2} are satisfied, and this yields the second short exact sequence in (2).
 Moreover, $Z_{\nn,m}=S_1\times \cdots \times S_r$ is aspherical,  
 so Proposition \ref{prop1part2} applies: this yields the first sequence.
  
  $F_{\nn,m} $ is a finite classifying space for its fundamental group, so by applying
  Lemma \ref{propbieri} to the first short exact sequence in (2) 
  we see that $\pi_1\overline{H}_{\nn,m}$ is of type $\mathcal{F}_n$  if and only if $\pi_1 H_{\nn,m}$ is of type $\mathcal{F}_n$.
  Theorem \ref{thmLlI1} tells us that $\pi _1 H_{\nn,m}$ is of type $\mathcal{F}_{r-1}$ but not of type $\mathcal{F}_r$. Thus (1) is proved.
  
  The holonomy representation of the fibration $\overline{H}_{\nn,m}\rightarrow H_{\nn,m}$ is the restriction 
  $$\nu=(\rho_1,\cdots,\rho_r)|_{\pi_1H_{\nn,m}}: \pi_1 H_{\nn,m} \rightarrow \mathrm{Out}(\pi_1 R_1')\times \cdots \times \mathrm{Out}(\pi_1 R_r')$$
 where $\rho_i$ is the holonomy of $X_{N_i,m}\to S_i$. Since the branched covering maps $f'_i$ are surjective on fundamental groups it follows from the short exact sequence induced by $h'$ that the projection of $\nu(\pi_1H)$ to $\mathrm{Out}(\pi_1 R_i')$ is $\rho_i(\pi_1 S_i)$. In particular, the map $\nu$ has infinite image in $\mathrm{Out}(\pi_1 F)$ as each of the $\rho_i$ do.  This proves (2).
 
 Assertion (3) follows immediately from (2) and the group theoretic Proposition \ref{p:not-commens} below.
 \end{proof}

\begin{remark}[Explicit presentations]
 The groups $\pi_1\overline{H}_{\nn,m}$
 constructed above are fibre products over $\ZZ^2$. Therefore, given finite presentations for the groups $\pi_1 X_{N_i,m}$, $1\leq i \leq r$, we could apply an algorithm developed by the first author, Howie, Miller and Short \cite{BriHowMilSho-13} to construct explicit finite presentations for our examples. 
 An implementation by the second author in a similar situation \cite{Llo-16} demonstrates the practical nature
 of this algorithm.
 \end{remark}

\section{Commensurability to direct products}
\label{sec:ProdCrit}

Each of the new K\"ahler groups $\G :=\pi_1\overline{H}$ constructed in
Theorems 1.1 and 1.2 fits into a short exact sequence of finitely generated groups
\begin{equation}\label{ses}
1\to \Delta \to \G \to Q\to 1,
\end{equation}
where $\Delta=\Sigma_1\times \cdots \times \Sigma_r$ is a product of $r\geq 1$ closed surface groups $\Sigma_i$
of genus $g_i\geq 2$. 

Such short exact sequences arise whenever one has a fibre bundle whose base $B$ has fundamental
group $Q$ and whose fibre $F$ is a product of surfaces: the short exact sequence is the beginning of the long exact sequence
in homotopy, truncated using the observation that since $\Delta$ has no non-trivial normal abelian subgroups,
the map $\pi_2B\to \pi_1F$ is trivial. For us, the fibration in question is
$\overline{H}\to H$, and (\ref{ses}) is a special case of the sequence in Proposition 2.3. In the setting of Theorem 1.1, the holonomoy representation
$Q\to {\rm{Out}}(\Delta)$ has infinite image, and in the setting of Theorem 1.2
it has finite image.

In order to complete the proofs of the theorems stated in the introduction, we must determine (i) when groups
such as $\G$ can be embedded in a product of surface groups, (ii) when they contain subgroups of finite index
that admit such embeddings, and (iii) when they are commensurable with residually free groups. In this
section we shall answer each of these questions.

Throughout this section we shall use the term \textit{surface group} to mean the fundamental group of a closed surface of negative Euler characteristic (which may be non-orientable).

\subsection{Residually free groups and limit groups}
 
A group $G$ is \textit{residually free} if for every element $g\in G\setminus \left\{1\right\}$  there is a free group $\F_r$ on $r$ generators and a homomorphism $\epsilon : G\rightarrow \F _r$ such that $\epsilon(g)\neq 1$.  A group $G$ is a \textit{limit group} (equivalently, \textit{fully residually free}) if for every finite subset 
$S \subset G$ there is a  homomorphism to a free group $\phi_S: G \rightarrow \F$ such
that the restriction of $\phi_S$ to $S$ is injective. 

It is easy to see that direct products of residually free groups are residually free. In contrast, the product of two or more non-abelian limit groups is not a limit group.

Limit groups are a fascinating class of groups that have been intensively studied in recent years at the confluence of geometry, group theory and logic \cite{Sel-01,KhaMya-98}. They admit several equivalent definitions, the equivalence of which confirms the
aphorism that, from many different perspectives, they are the natural class of ``approximately free groups".
A finitely generated group is residually free if and only if it is a subgroup of a direct product of finitely
many limit groups \cite[Corollary 19]{BauMyaRem-99} (see also p.4 and in particular Theorem C in \cite{BriHowMilSho-13}).

All surface groups are limit groups \cite{Bau-62} {\em{except}} 
 $\G_{-1}=\left\langle a,b,c\mid a^2b^2c^2\right\rangle $,
 the fundamental group of the non-orientable closed surface with euler characteristic $-1$, which 
 is not residually free: in a free group, any triple of elements satisfying the
equation $x^2y^2=z^2$ must commute \cite{Lyn-59}, so $[a,b]$ lies in the
kernel of every homomorphism from $\Gamma_{-1}$ to a free group.

\subsection{Infinite holonomy}

\begin{proposition} \label{p:not-commens}
If the holonomy representation $Q\to{\rm{Out}}(\Delta)$ associated to (\ref{ses})
has infinite image, then no subgroup of finite index in $\G$ is residually free, and therefore $\G$
is not commensurable with a subgroup of a direct product of surface groups.
\end{proposition}

\begin{proof}
Any automorphism of $\Delta=\Sigma_1\times \cdots \times \Sigma_r$ must leave the set of subgroups $\left\{\Sigma_1,\cdots,\Sigma_r\right\}$ invariant (cf. \cite[Prop.4 ]{BriMil-04}). Thus
${\rm{Aut}}(\Delta)$ contains a subgroup of finite index that leaves each
$\Sigma_i$ invariant and $\mathcal{O}={\rm{Out}}(\S_1)\times\dots\times
{\rm{Out}}(\S_r)$ has finite index in
${\rm{Out}}(\Delta)$. 

Let $\rho:Q\to{\rm{Out}}(\Delta)$ be the holonomy representation, let
$Q_0=\rho^{-1}(\mathcal{O})$, and let $\rho_i:Q_0\to {\rm{Out}}(\S_i)$
be the obvious restriction. If the image of $\rho$ is infinite, then the
image of at least one of the $\rho_i$ is infinite. Infinite subgroups
of mapping class groups have to contain elements of infinite order (e.g.\cite[Corollary 5.14]{Kap-01}), so it follows that $\G$ contains a subgroup
of the form $M=\S_i\rtimes_\alpha\Z$, where $\alpha$ has infinite order in
${\rm{Out}}(\S_i)$. If $\G_0$ is any subgroup of finite index in $\G$,
then $M_0=\G_0\cap M$ is again of the form $\S\rtimes_\beta\Z$, where $\S=\G_0\cap\S_i$ is a hyperbolic surface group and
 $\beta\in{\rm{Out}}(\S)$ (which is the restriction of a power of $\alpha$) has infinite
 order.

$M_0$ is the fundamental group of a closed aspherical
3-manifold that does not virtually split as a direct product, and therefore it
cannot be residually free, by Theorem A of \cite{BriHowMilSho-09}. As any subgroup of
a residually free group is residually free, it follows that $\G_0$ is not
residually free.

For the reader's convenience, we give a more direct proof that $M_0$
is not residually free. If it were, then by \cite{BauMyaRem-99} it would be
a subdirect product of limit groups $\Lambda_1\times\dots\times\Lambda_t$. Projecting away from factors that
$M_0$ does not intersect, we may assume that $\Lambda_i\cap M_0\neq 1$
for all $i$. As $M_0$ does not contain non-trivial normal
abelian subgroups, it follows that 
the $\Lambda_i$ are non-abelian. As limit groups
are torsion-free and $M_0$ does not contain $\Z^3$, it follows that $t\le 2$.
Replacing each
$\Lambda_i$ by the coordinate projection $p_i(M_0)$, we may assume that $M_0
<\Lambda_1\times\Lambda_2$ is a subdirect product (i.e. maps onto both
$\Lambda_1$ and $\Lambda_2$). Then, for $i=1,2$,
the intersection $M_0\cap\Lambda_i$ is normal in $\Lambda_i=p_i(M_0)$.
Non-abelian limit groups do not have non-trivial normal abelian subgroups,
so $I_i=M_0\cap\Lambda_i$ is non-abelian. But any non-cyclic subgroup
of $M_0$ must intersect $\S$, so $I_1\cap\S$ and $I_2\cap\S$ are infinite,
disjoint, commuting, subgroups of $\S$. This contradicts the fact that $\S$
is hyperbolic.
\end{proof}

\begin{corollary} The group $\pi_1\overline{H}$ constructed in Theorem 1.1 is not commensurable with a subgroup of a direct product of surface groups.
\end{corollary}

\subsection{Finite holonomy}
When the holonomy $Q\to{\rm{Out}}(\Delta)$
is finite, it is easy to see that $\G$ is virtually a direct product.

\begin{proposition}
In the setting of (\ref{ses}), if the holonomy representation $Q\to{\rm{Out}}(\Delta)$ is finite, then $\G$ has a subgroup of finite index that is
residually free {\rm{[}}respectively, is a subgroup of a direct
product of surface groups{\rm{]}} if and only if $Q$ has such a subgroup of 
finite index.
\label{p:infholo}
\end{proposition}

\begin{proof} 
Let $Q_1$ be the kernel of $Q\to{\rm{Out}}(\Delta)$
and let $\G_1<\G$ be the inverse image of $Q_1$.
Then, as the centre of $\Delta$ is trivial, $\G_1 \cong \Delta\times Q_1$.

Every subgroup of a residually free group is residually free, and the direct product of residually free groups is residually free. Thus
the proposition follows from the fact that orientable surface groups are residually
free.
\end{proof}

\subsection{Residually-Free K\"ahler groups}

We begin with a non-trivial example of a Kodaira fibration whose
fundamental group is residually-free.

\begin{example}
Let $G$ be any finite group and for $i=1,2$ let $q_i:\Sigma_i\to G$ be an
epimorphism from a hyperbolic surface group $\Sigma_i=\pi_1S_i$. Let $P<\Sigma_1\times\Sigma_2$
be the fibre product, i.e. $P=\{(x,y)\mid q_1(x)=q_2(y)\}$. The projection onto the
second factor $p_2: P\rightarrow \Sigma _2$ induces a short exact sequence 
\[
 1 \rightarrow \Sigma' _1 \rightarrow P \rightarrow \Sigma_2\rightarrow 1
\]
with $\Sigma'_1=\ker q_1\unlhd \Sigma_1$ a finite-index normal subgroup. The action of $P$ by conjugation on $\Sigma'_1$ defines a homomorphism $\Sigma_2\to
{\rm{Out}}(\Sigma_1')$ that factors through
 $q_2: \Sigma_2 \rightarrow G=\Sigma_1/\Sigma'_1$. 
 
Let $S_1'\to S_1$ be the regular covering of $S_1$ corresponding to $\Sigma_1' \unlhd \Sigma_1$.
Nielsen realisation \cite{Ker-83} realises the action of $\Sigma_2$ on
$\Sigma_1'$ as a group of diffeomorphisms of $S_1'$, and thus we obtain a
smooth surface-by-surface bundle $X$ with $\pi_1 X=P$,
that has fibre $S'_1$, base $S_2$ and holonomy
representation $q_2$. Theorem \ref{thm3} and Addendum \ref{corkot} imply that $X$ can be endowed with the structure of a Kodaira fibration.  
\end{example}

Our second example illustrates the fact that torsion-free
K\"ahler groups that are virtually residually-free need not be residually-free.

\begin{example}  Let $R_g$ be a closed orientable surface of genus $g$ and imagine it as the connected sum of $g$ handles placed in cyclic order around a sphere. We
consider the automorphism that rotates this picture through $2\pi/g$. Algebraically,
if we fix the usual
presentation
$\pi_1 R_g=\left\langle \alpha_1,\beta_1,\cdots, \alpha_g,\beta_g\mid \left[\alpha_1,\beta_1\right]\cdots \left[\alpha_g,\beta_g\right]\right\rangle$,
this rotation (which has two fixed points)
defines an automorphism $\phi$ that sends $\alpha_i\mapsto \alpha_{i+1}$, $\beta_i\mapsto \beta_{i+1}$ for $1\leq i\leq g-1$ and $\alpha_g\mapsto \alpha_1$, $\beta_g\mapsto \beta _1$. Thus $\left\langle \phi\right\rangle \leq \aut(\pi_1 R_g)$ is a cyclic subgroup of order $g$. 

Let $T_h$ be an arbitrary closed surfaces of genus $h\geq 2$ and let $\rho: \pi_1 T_h\to \left\langle \overline{\phi} \right\rangle\cong \ZZ/g\ZZ \leq \out(\pi_1 R_g)$ be the map defined by sending each element of a standard symplectic basis for $H_1(\pi_1 T_h,\ZZ)$ to $\overline{\phi}:=\phi \cdot \inn(\pi_1 R_g)$. 
Consider a Kodaira fibration $R_g\hookrightarrow X'\rightarrow T_h$ with holonomy $\rho$. It follows from Proposition  \ref{propFinHol}
that $\pi_1 X'$ is not residually free.
And it follows from  Theorem \ref{t:resfree} that if the Kodaira fibrations in Theorem \ref{thm4} are of this form then the K\"ahler group $\pi_1 \overline{H}$ is not residually free.
\end{example}

\begin{lemma}  
Let $S$ be a hyperbolic surface group and let $G$ be a group that contains $S$ as a normal subgroup.
The following conditions are equivalent:

\begin{enumerate}[(i)]
\item the image of the map $G\to\aut(S)$ given by conjugation is torsion-free and the image of $G\to\out(S)$ is finite;

\item one can embed $S$ as a normal subgroup of finite index in a surface group $\S$ so that $G\to\aut(S)$ factors through $\inn(\S)\to \aut(S)$. 
\end{enumerate}
\label{lemFinHolOneFact}
\end{lemma}

\begin{proof} 
If (i) holds then the image $A$ of $G\to\aut(S)$ is torsion free and contains $\inn(S)\cong S$ as a subgroup of finite index. A torsion-free
finite extension of a surface group is a surface group, so we can define $\S=A$.
The converse follows immediately from the fact that centralisers of non-cyclic subgroups in hyperbolic surface groups are trivial.
\end{proof}

Lemma \ref{lemFinHolOneFact} has the following geometric interpretation,
in which $\Sigma$ emerges as $\pi_1(\widetilde{R}/\Lambda)$.

\begin{addendum}
 With the hypotheses of Lemma \ref{lemFinHolOneFact}, let $R$ be a closed 
  surface with $S=\pi_1 R$, let $\L$ be the image of $G\to \aut(S)$ and let $\overline{\L}$ be the image of $G\to \out(S)$. Then conditions (i) and (ii) 
  are equivalent to the geometric condition that $\overline{\L}$ is finite and the action $\overline{\L}\to \mathrm{Homeo}(R)$ given by Nielsen realisation is free.
\end{addendum}
\begin{proof}
Assume that condition (i) holds. Since $\overline{\L}$ is finite,
Kerckhoff's solution to the Nielsen realisation problem \cite{Ker-83} enables us
to realise $\L$ as a cocompact Fuchsian group: $\overline{\L}$ can be realised as a group of isometries of a hyperbolic metric $g$ on $R$ and $\L$ is the discrete group of isometries of the universal cover $\widetilde{R}\cong \H^2$ consisting of all lifts of $\overline{\L}\leq \isom(R,g)$. As a Fuchsian group, $\L$ is torsion-free 
if and only if its action on $\widetilde{R}\cong \H^2$ is free, and this is the
case if and only if the action of $\overline{\L}=\L/S$ on $R$ is free.
\end{proof}

\begin{proposition} Consider a short exact sequence $1\to F \to G \to Q \to 1$, where $F$ is a direct product of finitely many hyperbolic surface groups $S_i$, each of which is
normal in $G$, and $Q$ is torsion-free. The following conditions are equivalent:

\begin{enumerate}[(i)]
\item $G$ can be embedded in a direct product of surface groups;

\item $Q$ can be embedded in a direct product of surface groups and the image of each of the conjugation actions
$G\to\aut(S_i)$ is torsion-free and has finite image in $\out(S_i)$.
\end{enumerate}
\label{propFinHol}
\end{proposition}

\begin{proof}

 If (ii) holds then by Lemma \ref{lemFinHolOneFact} there are surface
 groups $\S_i$ with $S_i\unlhd \S_i$ of finite index such that the map $G\to \aut(S_i)$ given by conjugation factors through $G\to \inn(\S_i)\cong\S_i$. We combine these maps with the composition of $G\to Q$ and the embedding of $Q$ to obtain a map $\Phi$ from $G$ to a product of surface groups. The kernel of the map $G\to Q$ is the product of the $S_i$, and each $S_i$ embeds into the coordinate for $\S_i$, so $\Phi$ is injective and (i) is proved.

We shall prove the converse in the surface group case; the other case is entirely
similar. Thus we   assume that $G$ can be embedded in a direct product $\S _1 \times \cdots \times \S_m$ of surface groups.
After projecting away from factors $\S_i$ that have trivial intersection with $G$ and replacing the $\S_i$ with the coordinate projections of $G$,
we may assume that $G\leq \L_1 \times \cdots \times \L_m$ is a full subdirect product, where each $\L_i$ is either a surface group, a nonabelian free group, or $\ZZ$. Note that $G\cap \L_i$ is normal in $\L_i$, since it is normal in $G$
and $G$ projects onto $\L_i$.

By assumption $F=S_1\times \cdots \times S_k$ for some $k$. We want to show that, after reordering factors, $S_i$ is a finite-index normal subgroup of $\L_i$. Denote by $p_i:\L_1\times \cdots \times\L_m\rightarrow \L_i$ the projection onto the $i$th factor. Since each $S_j$ is normal in the subdirect product $G\leq \L_1\times \cdots \times \L_m$ the projections $p_i(S_j)\unlhd \L_i$ are finitely-generated normal subgroups for $1\leq i\leq m$. Since the $\L_i$ are surface groups or free groups, it
follows, each $p_i(S_j)$ is either trivial or of finite index.

Since $S_j$ has no centre, it intersects abelian factors trivially. Suppose $\L_i$
is non-abelian. We claim that if $p_i(S_j)$ is nontrivial, then $S_j\cap \L_i$ is nontrivial. If this were not the case, then the normal subgroups $S_j$ and
$G\cap\L_i$ would intersect trivially in $G$, and hence would commute. But 
this is impossible, because the centraliser in $\L_i$ of the
finite-index subgroup $p_i(S_j)$ is trivial. Since  $S_j$ does not contain $\ZZ^2$, it can intersect only one $\L_i$. After reordering we may assume $i=j$ and that the projection $p_j(S_i)\leq \L_j$ is trivial for $j\neq i$. 

It follows that $F$ is a finite index subgroup of $\L_1 \times \dots \times \L_k$ with $S_i = F\cap \L_i$. In particular, $\L_i$ must be a surface group, and 
the action of $G$ by conjugation on $S_i$ factors through  ${\rm{Inn}}(\L_i)
\to {\rm{Aut}}(S_i)$.

Finally, we claim that $Q$ embeds in $\L_{k+1}\times \dots \times \L_m$. To see this note that since $S_i$ is normal in $\L_i$, $F$ is normal in $\L_1\times \dots \times \L_k$ with finite quotient $U$, say. Thus $Q= G/F$ embeds in $U\times \L_{k+1}\times \dots \times \L_m$, and since $Q$ is torsion-free it intersects $U$ trivially.
\end{proof}

\begin{addendum}
Under the hypotheses of Proposition \ref{propFinHol} the following conditions are equivalent:
\begin{itemize}
\item[(i')] $G$ can be embedded in a direct product of non-abelian limit groups and $\G_{-1}$;

\item[(ii')] $Q$ can be embedded in a direct product of non-abelian limit groups and $\G_{-1}$, and the image of each of the maps
$G\to\aut(S_i)$ is torsion-free and has finite image in $\out(S_i)$.
\end{itemize}
Moreover, if these conditions hold and $Q$ is a subgroup of a direct product of surface groups, then $G$ is also a subgroup of a direct product of surface groups.
\label{addLimit}
\end{addendum}

\begin{proof}
 Following the preceding proof, one sees that $\L_1, \dots, \L_k$ must still be surface groups, because a torsion-free group that contains a surface group of finite index is itself a surface group and any non-trivial finitely generated normal subgroup of a limit group must be of finite index (see \cite[Theorem 3.1]{BriHow-07}).
\end{proof}

\begin{theorem} Let the Kodaira fibrations  $S_{\gamma_i}\hookrightarrow
X_i\to S_{g_i}$ with zero signature be as in the statement of
Theorem \ref{thmSgn0Intro} and assume that each of the maps $\alpha_i:S_{g_i}\to E$ is surjective
on $\pi_1$. Then the following conditions are equivalent:
\begin{enumerate}
\item the K\"ahler group $\pi_1 \overline{H}$ is a subgroup of a direct product of limit groups and $\G_{-1}$;
\item $\pi_1 \overline{H}$ can be embedded in
a direct product of surface groups;
\item each  $\pi_1 X_i$ can be embedded in
a direct product of surface groups; 
\item for each $X_i$, the image of the homomorphism
 $\pi_1X_i \to\aut(\pi_1S_{\gamma_i})$ defined by conjugation is torsion-free.
\end{enumerate}
\label{t:resfree}
\end{theorem}

\begin{proof} The equivalence of (1) and (2) is covered by Addendum \ref{addLimit}. Proposition \ref{propFinHol} establishes the equivalence of (3)
and (4), and (2) is a trivial consequence of (3), so we concentrate on proving that
(2) implies (3). Assume that $\pi_1\overline{H}$ is a subgroup of a direct
product of surface groups.

The fibre of $X=X_1\times \cdots \times X_r\rightarrow S_{g_1}\times \dots \times S_{g_r}$ is $F= S_{\gamma_1}\times \dots \times S_{\gamma_r}$, and 
the restriction
of the fibration gives  $F\hookrightarrow \overline{H} \rightarrow H$.
Each $\pi_1S_{\g_i}$ is normal in both $\pi_1X$ and $\pi_1\overline{H}$. By
 Proposition \ref{propFinHol}, if (2) holds then the image of each of the maps $\phi_i: \pi_1 \overline{H}\to \aut (\pi_1 S_{\g_i})$ 
 given by conjugation is torsion-free, and the image in $\out (\pi_1 S_{\g_i})$
 is finite. The map $\phi_i$ factors
 through $\rho_i: \pi_1 X_i\rightarrow \aut(\pi_1 S_{\g_i})$. Because
  $\pi_1 H\leq \pi_1 S_{g_1}\times \cdots \times \pi_1 S_{g_r}$ is {\em{subdirect}}
  (i.e. maps onto each $S_{g_i}$), the group $\pi_1 \overline{H}\leq \pi_1 X_1 \times \dots \times \pi_1 X_r$ is also subdirect and the image of $\phi_i$ coincides with the image of $\rho_i$. 
  Therefore, the conditions of Proposition \ref{propFinHol} hold
  for each of the fibrations $S_{\gamma_i}\hookrightarrow
X_i\to S_{g_i}$. 
\end{proof}

\bibliography{References}
\bibliographystyle{amsplain}

\end{document}